\documentclass[a4paper,12pt]{amsart}
\usepackage{ amsmath, tabularx, amsthm, amssymb, amsfonts}
\usepackage{xcolor}
\definecolor{leaf}{rgb}{0,.35,0}
\definecolor{chianti}{rgb}{0.6,0,0}
\definecolor{meretale}{rgb}{0,0,.6}
\usepackage[inner=2.3cm,outer=2.3cm, bottom=3.3cm]{geometry}
\usepackage{setspace}
\usepackage[colorlinks=true,linkcolor=meretale,
urlcolor=chianti, citecolor=leaf,pagebackref]{hyperref}

\usepackage[all]{xy}  
\usepackage{comment}
\usepackage{enumerate}
\usepackage{mathrsfs}              
\usepackage{stmaryrd}
\usepackage{colonequals}
\usepackage{cleveref}

\theoremstyle{definition}
\newtheorem{theorem}{Theorem}[section]

\newtheorem{corollary}[theorem]{Corollary}
\newtheorem{lemma}[theorem]{Lemma}
\newtheorem{proposition}[theorem]{Proposition}

\theoremstyle{definition}
\newtheorem{definition}[theorem]{Definition}

\newtheorem{example}[theorem]{Example}

\newtheorem{remark}[theorem]{Remark}
\numberwithin{equation}{subsection}

\newtheorem{theoremx}{Theorem}

\newcommand{\m}{\mathfrak{m}}

\newcommand{\NN}{\mathbb{N}}
\newcommand{\ZZ}{\mathbb{Z}}

\newcommand{\FF}{\mathbb{F}}

\newcommand{\Hom}{\operatorname{Hom}}

\renewcommand{\a}{\mathfrak{a}}

\newcommand{\n}{\mathfrak{n}}

\newcommand{\lr}[1]{{\langle {#1} \rangle}}

\newcommand{\dif}[2]{^{\lr{#1}_{#2}}}
\newcommand{\cidr}[1]{^{\left\{{#1}\right\}}}
\newcommand{\derp}[1]{^{\lr{#1}_{\delta}}}
\newcommand{\difM}[1]{^{\lr{#1}_{\mathrm{mix}}}}
\newcommand{\difMV}[1]{^{\lr{#1}_{{\mathrm{mix}},V}}}

\renewcommand{\geq}{\geqslant}
\renewcommand{\leq}{\leqslant}

\author[De Stefani]{Alessandro De Stefani}
\address{Dipartimento di Matematica, Universit{\`a} di Genova, Via Dodecaneso 35, 16146 Genova,~Italy}
\email{alessandro.destefani@unige.it}

\author[Grifo]{Elo\'isa Grifo}
\address{Department of Mathematics, University of Nebraska, Lincoln, NE 68588-0130, USA}
\email{grifo@unl.edu}

\author[Jeffries]{Jack Jeffries}
\address{Department of Mathematics, University of Nebraska, Lincoln, NE 68588-0130, USA}
\email{jack.jeffries@unl.edu}

\begin{document}
\newcommand{\tens}{\otimes}
\newcommand{\hhtest}[1]{\tau ( #1 )}
\renewcommand{\hom}[3]{\operatorname{Hom}_{#1} ( #2, #3 )}
\renewcommand{\t}{\operatorname{type}}

\title[Differential and symbolic powers of ideals]{Differential and symbolic powers of ideals}

\begin{abstract}
We characterize symbolic powers of prime ideals in polynomial rings over any field in terms of $\ZZ$-linear differential operators, and of prime ideals in polynomial rings over complete discrete valuation rings with a $p$-derivation $\delta$ in terms of $\ZZ$-linear differential operators and of $\delta$. This extends previous work of the same authors, as it allows the removal of separability hypotheses that were otherwise necessary. The absence of separability and the fact that modules of $\ZZ$-linear differential operators are typically not finitely generated require the introduction of new techniques. As a byproduct, we extend a characterization of symbolic powers due to Cid-Ruiz which also holds in the nonsmooth case. Finally, we produce an example of an unramified discrete valuation ring that has no $p$-derivations.
\end{abstract}

\maketitle

\section{Introduction}
Given a prime ideal $Q$ in a noetherian ring $R$, its powers $Q^n$ might have embedded primes. The $n$th symbolic power of $Q$ is the minimal component of $Q^n$, obtained by deleting all the embedded components:
\[ Q^{(n)} = Q^n R_Q \cap Q. \]
Zariski \cite{Zariski} and Nagata \cite{Nagata} showed that when $R = K[x_1, \ldots, x_d]$ is a polynomial ring over a field $K$, then
\[ Q^{(n)} = \! \! \! \! \! \! \! \bigcap_{\substack{\m \supseteq Q \\ \m \text{ maximal ideal}}} \! \! \! \! \! \! \! \m^n.\]
When the field is perfect, it is well-known that this can be rewritten in terms of differential operators \cite[Proposition 2.14]{SurveySP}: denoting the $K$-linear differential operators on $R$ of order up to $n$ by $D_{R|K}^n$ (see \Cref{diff ops def}), 
\[ Q^{(n)} = \{ f \in R \mid \partial(f) \in Q \text{ for all } \partial \in D^{n-1}_{R|K} \}.\]
For a general noetherian ring $R$, the set
\[ Q\dif{n}{K} \colonequals \{ f \in R \mid \partial(f) \in Q \text{ for all } \partial \in D^{n-1}_{R|K} \} \]
is a $Q$-primary ideal known as the $n$th differential power of $Q$, which always contains the $n$th symbolic power~$Q^{(n)}$. The differential version of Zariski-Nagata says that the symbolic and ordinary powers of $Q$ coincide when $R$ is a polynomial ring over a perfect field.

The assumption that $K$ is perfect is necessary. For example, when $K=\mathbb{F}_p(t)$ and $R = K[x]$, and $Q = (x^p-t)$, we have $Q^{(n)} = Q^n$ for all $n$, while $D^1_{R|K} = R \oplus R \cdot \frac{\partial}{\partial x}$ satisfies
\[ D^1_{R|K} \left( x^p-t \right) \in Q, \]
so
\[ x^p-t \in Q\dif{2}{K} \smallsetminus Q^{(2)}.\]
The issue is we need \emph{more} differential operators to be able to see that $x^p-t \notin Q^{(2)}$. In this paper, we show that the differential version of Zariski-Nagata holds in full generality if we consider all $\ZZ$-linear differential operators.

\begin{theoremx}[see \Cref{thm general field}]
Let $K$ be any field and $R = K[x_1, \ldots, x_d]$. For any prime ideal $Q$ in $R$ and any $n \geqslant 1$,
\[ Q^{(n)} = \{ f \in R \mid \partial(f) \in Q \text{ for all } \partial \in D^{n-1}_{R|\ZZ} \}.\]
\end{theoremx}

In the example of $Q = (x^p-t)$ over $K=\mathbb{F}_p(t)$, we now have access to the $\ZZ$-linear differential operator $\frac{\partial}{\partial t}$, which does see the job of excluding $x^p-t$:
\[ \frac{\partial}{\partial t}(x^p-t) \notin Q \implies x^p-t \notin Q\dif{2}{\ZZ}.\]

If instead of a field $K$ we consider a polynomial ring $R = V[x_1, \ldots, x_d]$ where $V$ is an unramified discrete valuation ring with uniformizer $p$, there are two different types of prime ideals: the primes that contain $p$ and the ones that do not. The primes that do not contain behave very much like the field case, and their symbolic powers are given by \mbox{$A$-linear} differential powers \cite[Theorem~3.9]{ZNmixed}. However, we cannot describe the symbolic powers of primes containing $p$ using only differential operators, since any differential operator is in particular additive. For example, when $R = \ZZ_p[x]$ and $Q = (p)$, we have $Q^{(n)} = Q^n$ for all $n$, but any $\ZZ$-linear differential operator $\partial$ of any order will satisfy
\[ \partial(p) = p \partial(1) \in Q \implies p \in Q\dif{n}{\ZZ} \smallsetminus Q^n \text{ for all } n \geqslant 2.\]
The key ingredient missing is that we need an operator that can decrease $p$-adic order. In \cite{ZNmixed}, we showed that the missing ingredient is a $p$-derivation; see \Cref{p-derivation definition}. While not all rings have a $p$-derivation, any polynomial ring over a complete unramified discrete valuation ring does. Even in this case, in \cite{ZNmixed} we needed to impose some separability assumptions, which we now remove.

\begin{theoremx}[see \Cref{main}]
Let $V$ be a complete discrete valuation ring with uniformizer $p>0$, and $R=V[x_1,\ldots,x_n]$. Let $\delta$ be a $p$-derivation on $R$. If $Q$ is a prime ideal of $R$ that contains $p$, then 
\[ Q^{(n)} = \{f \in R \mid (\delta^a \circ \partial)(f) \in Q \text{ for all } \partial \in D_{R|\ZZ}^b \text{ with } a+b \leqslant n-1\}.\]
\end{theoremx}

We call the set on the right the $n$th mixed differential power of $Q$.

These new results and the classical proofs involve proving both containments separately. First one shows that the $n$th (mixed) differential power of $Q$ contains $Q^{(n)}$, which is elementary. To show that the symbolic power must contain the (mixed) differential power, in the classical setting one reduces the problem by localizing at $Q$, and showing two key facts:
\begin{itemize}
    \item[Step 1] The symbolic and (mixed) differential powers of the maximal ideal (in the appropriate localization) coincide.
    \item[Step 2] Taking (mixed) differential powers commutes with localization.
\end{itemize}

In the absence of a separability assumption, the first step may not hold for $K$-linear (or $V$-linear) differential operators, and one needs instead to consider $\ZZ$-linear differential operators. The second step involves proving that differential operators localize well, which no longer holds for $\ZZ$-linear differential operators. The main issue relates to the module of principal parts $P^n_{R|\ZZ}$, which has the property that $D^n_{R|\ZZ} \cong \Hom_R(P^n_{R|\ZZ}, R)$. If $P^n_{R|\ZZ}$ were finitely generated (which holds for $P^n_{R|K}$ and $P^n_{R|V}$), one can easily show that differential operators localize. But $P^n_{R|\ZZ}$ is typically not finitely generated, which is the main challenge that we need to overcome to prove \Cref{thm general field} and \Cref{main}.
In the equicharacteristic case, we also show that we can upgrade from $\ZZ$-linear differential operators to $K_0$-linear differential operators, where $K_0$ is a subfield of $K$ over which the modules of principal parts and differential operators are finitely generated $R$-modules.

Over a singular ring, the (mixed) differential powers cannot describe all symbolic powers. In fact, Brenner, the third author, and N\'u\~nez Betancourt showed that over any singular domain that is a finitely generated algebra over a perfect field, there exists a maximal ideal whose symbolic and differential powers do not coincide \cite[Theorem~10.2]{DiffSig}.

Cid-Ruiz \cite{Yairon} showed that in the case of singular rings, one can instead use differential operators $R \longrightarrow R/Q$ to describe the symbolic powers of $Q$; see also \cite[Proposition~1.5]{Brumfiel}. But as in the regular case, the main result needed an analogous separability assumption, which we now remove by using $\ZZ$-linear operators.

\begin{theoremx}[see \Cref{thm Yairon}]
Let $K$ be a field, and $R$ be a $K$-algebra essentially of finite type over $K$. For all prime ideals $Q$ of $R$,
\[
Q^{(n)} = \{f \in R \mid \partial(f) = 0 \text{ for all } \partial \in D^{n-1}_{R|\ZZ}(R,R/Q)\}.
\]
\end{theoremx}

We also illustrate in \Cref{example Yairon mixed char} why there is no hope of such a theorem in mixed characteristic for prime ideals containing $p$.


Finally, in \Cref{bad dvr} we give an example of an unramified discrete valuation ring that has no $p$-derivations; for contrast, note that a complete unramified discrete valuation ring always has a $p$-derivation.

\section{Preliminaries}

\subsection{Symbolic powers} 
Let $R$ be a commutative ring with unity, and $I$ be an ideal in $R$. Let $W$ be the complement in $R$ of the minimal primes of $I$. For $n \geq 1$, the $n$th symbolic power of $I$ is defined as 
\[
I^{(n)} = I^nR_W \cap R.
\]
When $I=Q$ is a prime ideal, $Q^{(n)}$ coincides with the smallest $Q$-primary ideal containing~$Q^n$, i.e., with the $Q$-primary component of~$Q^n$.

A classical theorem of Zariski and Nagata \cite{Zariski,Nagata} identifies the $n$th symbolic power of a prime $Q \subseteq \mathbb{C}[x_1,\ldots,x_d]$ as the ideal of all functions that vanish up to order~$n$ along the algebraic variety defined by~$Q$. This result was extended first to the case of polynomial rings over a perfect field \cite{SurveySP}, and then to the case of polynomial rings of mixed characteristic \cite{ZNmixed}. In order to properly state these results, we need to recall the notion of differential operators, which we do next.

\subsection{Differential powers}
Let $A$ be a commutative ring with unity, and $R$ be an $A$-algebra. 

\begin{definition}\label{diff ops def}
The module of {\bf $A$-linear differential operators} on $R$ of order at most $n$ is the $R$-module $D^n_{R|A} \subseteq \Hom_A(R,R)$ defined inductively as follows:
\begin{itemize}
\item $D^0_{R|A} = \Hom_R(R,R)$.
\item $D^n_{R|A} = \{\partial \in \Hom_A(R,R) \mid [\partial,r] \in D^{n-1}_{R|A}$ for every $r \in R\}$.
\end{itemize}
Here $[\partial,r]$ denotes the commutator $[\partial,r] \colonequals \partial \circ \mu_r - \mu_r \partial$, with~$\mu_r$ the multiplication by~$r$.

More generally, given $R$-modules $M$ and $N$, and an element $r \in R$, let $\mu_{r,M}$ denote the map given by multiplication by $r$ on $M$. The module $D^n_{R|A}(M,N) \subseteq \Hom_A(M,N)$ of {\bf $A$-linear differential operators of order up to $n$} from $M$ to $N$ is defined as follows:
\begin{itemize}
\item $D^0_{R|A}(M,N) = \Hom_R(M,N)$.
\item $D^n_{R|A}(M,N) = \{\partial \in \Hom_A(M,N) \mid \partial \circ \mu_{r,M} - \mu_{r,N} \circ \partial \in D^{n-1}_{R|A}(M,N) $ for all $r \in R\}$.
\end{itemize}
The module of {\bf $A$-linear differential operators} from $M$ to $N$ is 
$$D_{R|A}(M,N) \colonequals \bigcup_{n \geqslant 0} D^n_{R|A}(M,N).$$ 
\end{definition}

Given an ideal $I$ in $R$, the {\bf $n$th ($A$-linear) differential power} of $I$ is
$$I\dif{n}{A} \colonequals \{f \in R \mid \partial(f) \in I \text{ for all } \partial \in D^{n-1}_{R|A}\}.$$

This set $I^{\langle n \rangle_A}$ is an ideal of $R$ \cite[Proposition 2.4]{SurveySP}. Moreover, if $I$ is a $Q$-primary ideal, then so is $I^{\langle n \rangle_A}$ for every $n \geqslant 1$ \cite[Proposition 2.6]{SurveySP}. In particular, $I^{(n)} \subseteq I^{\langle n \rangle_A}$ \cite[Proposition 3.2]{ZNmixed}. See also \cite{Brumfiel}.

\subsection{Modules of principal parts}

Let $R$ be any $A$-algebra. We set $P_{R|A} \colonequals R\otimes_A R$ and let $\mu_{R|A}\!: P_{R|A} \to R$ denote the multiplication map induced by $a \otimes b \mapsto ab$. Consider the kernel $\Delta_{R|A} \colonequals \ker(\mu_{R|A})$, and the quotient module 
\[ P^n_{R|A} \colonequals P_{R|A} / \Delta^{n+1}_{R|A}.\]
More generally, for an $R$-module $M$, we set $P_{R|A}(M) = R\otimes_A M$ and 
\[ P^n_{R|A}(M) \colonequals P_{R|A}(M) / \Delta^{n+1}_{R|A} P_{R|A}(M).\] 
In particular, $P^n_{R|A} = P^n_{R|A}(R)$. We consider $P^n_{R|A}(M)$ as an $R$-module via the action 
$$r \cdot \overline{s\otimes m} \colonequals \overline{rs\otimes m}.$$

The {\bf universal differential operator} of order $n$ from $M$ is the map 
\[d:\! M \to P^n_{R|A}(M) \qquad \textrm{given by} \quad d(m) = \overline{1 \otimes m}.\]
Any other differential operator of order at most $n$ factors through $d$: for any $R$-module $N$ and any differential operator $\partial \in D^n_{R|A}(M,N)$, there exists a unique $R$-linear map $\phi\in \Hom_R(P^n_{R|A}(M),N)$ such that $\partial = \phi \circ d$. In fact, setting $\pi \!: R \otimes_A M \to P^n_{R|A}(M)$ to be the canonical quotient map, the map
$$\Hom_R(P^n_{R|A}(M), N) \to D^n_{R|A}(M,N)$$
induced by the composition
$$\Hom_R(P^n_{R|A}(M), N) \xrightarrow{\pi^*} \Hom_R(R \otimes_A M, N) \xrightarrow{d^*} D^n_{R|A}(M,N)$$
is an isomorphism of modules over $R \otimes_A R$.

\begin{lemma}\label{lemma powers moving around in tensor}
	Let $A$ be any ring and $R$ be an $A$-algebra. Given an ideal $I$ in $R$, and any $n \geqslant t \geqslant 1$,
	$$(R \otimes_A I^n) P^{t}_{R | A} \subseteq (I^{n-t} \otimes_A R) P^{t}_{R | A}.$$
\end{lemma}

By convention, $I^n = R$ whenever $n \leqslant 0$.

\begin{proof}
We will show that in the ring $R\otimes_A R$ one has
\[ R \otimes_A I^n + \Delta^{t+1}_{R|A} \subseteq \sum_{j=1}^{t+1} I^j \otimes_A I^{n-j} + \Delta^{t+1}_{R|A} ,\]
from which it follows by a straightforward induction on $n$ that
\[ R \otimes_A I^n + \Delta^{t+1}_{R|A} \subseteq I^{n-t} \otimes_A R + \Delta^{t+1}_{R|A}.\]
The statement then follows by taking homomorphic images in the principal parts.

To show the claim, first note that $(R \otimes_A I^n) P^t_{R | A}$ is generated by the classes of simple tensors of the form 
	$$r \otimes (f_1 \cdots f_n) \qquad \text{for } r \in R, f_i \in I.$$
	Moreover,
	$$r \otimes (f_1 \cdots f_n) = (r \otimes g) (1 \otimes f_1) \cdots (1 \otimes f_{t+1}),$$
	with $g = f_{t+2} \cdots f_n \in I^{n-t-1}$,
	and expanding $(1 \otimes f_1 - f_1 \otimes 1) \cdots (1 \otimes f_{t+1} - f_{t+1} \otimes 1) \in \Delta^{t+1}_{R|A}$ shows that
$$(1 \otimes f_1 - f_1 \otimes 1) \cdots (1 \otimes f_{t+1} - f_{t+1} \otimes 1) - (1 \otimes f_1) \cdots (1 \otimes f_{t+1}) \in \sum_{j=1}^{t+1} I^j \otimes_A I^{t+1-j}.$$
	Therefore,
\[\begin{aligned}(r \otimes g) (1\otimes f_1) (1 \otimes f_2) \cdots (1 \otimes f_{t+1}) + \Delta^{t+1}_{R|A} &\in (R\otimes_A I^{n-t-1}) \left(\sum_{j=1}^{t+1} I^j \otimes_A I^{t+1-j}\right)+ \Delta^{t+1}_{R|A}\\
&\subseteq \sum_{j=1}^{t+1} I^j \otimes_A I^{n-j} + \Delta^{t+1}_{R|A}\end{aligned}\]
	as claimed.
\end{proof}

\begin{lemma} \label{lemma containment diff}
	Let $A$ be any ring, $R$ an $A$-algebra, and $M$ an $R$-module. Given an ideal $I$ in $R$ and any $n \geqslant t \geqslant 1$, every differential operator $\partial \in D^t_{R|A}(R,M)$ satisfies
	$$\partial(I^n) \subseteq I^{n-t} M.$$
\end{lemma}

\begin{proof}
	Let $d\!: R \longrightarrow P^t_{R|A}$ denote the universal differential operator of order at most $t$. For any $a \in I^n$,
	$$d(a) = 1 \otimes a \in (R \otimes_A I^n) P^t_{R | A}.$$		
	Any differential operator $\partial \in D^t_{R|A}(R,M)$ can be written as a composition $\partial = \varphi \circ d$ for some $\varphi \in \Hom_R(P^t_{R|A},M)$. Thus
\[ 
\begin{aligned}	
\partial(I^n) & = \varphi(d(I^n)) \\
	& \subseteq \varphi((R \otimes_A I^n) P^t_{R | A}) & \\
	& \subseteq \varphi ( (I^{n-t} \otimes_A R) P^t_{R | A} ) & \textrm{by \Cref{lemma powers moving around in tensor}} \\
	& \subseteq I^{n-t} \varphi ( P^t_{R | A} ) & \textrm{since } \varphi \textrm{ is } R\textrm{-linear} \\ 
	& \subseteq I^{n-t} M 
	\end{aligned}
	\]
    concluding the proof.
\end{proof}

\begin{lemma}\label{module of principal parts direct sum}
Let $A\to B$ be a ring homomorphism and $S = B[x_1, \ldots, x_d]$.
For any positive integer $n$, the module $P^n_{R|A}$ decomposes as an $R$-module direct sum of copies of $P^i_{B|A}$ for $i\leqslant n$. Moreover, if $B$ is a field, then $P^n_{R|A}$ is free, and so is $P^n_{R_Q|A}$ for any prime ideal $Q$ of~$R$.
\end{lemma}

\begin{proof}
We have 
\[\begin{aligned}
P^n_{R|A}&\cong \frac{(B\otimes_A B)[x_1,\dots,x_d,y_1,\dots,y_d]}{\Delta_{B|A} + (y_1-x_1,\dots,y_d-x_d)^{n+1}}\\
& \cong \frac{(B\otimes_A B)[x_1,\dots,x_d,y_1,\dots,y_d]}{\sum_{i=0}^{n} \Delta_{B|A}^{i+1} (y_1-x_1,\dots,y_d-x_d)^{n-i}}\end{aligned}\]
Write $z_i= y_i-x_i$. We can give $P^n_{R|A}$ an $\NN$-grading by setting the degree of $z_i$ to be one and the degree of every element in $(B\otimes_A B)[x_1,\dots,x_d]$ to be zero. Then $P^n_{R|A}$ is generated over $(B\otimes_A B)[x_1,\dots,x_d]$ by the monomials in $z$ of degree at most $n$, and the annihilator of a monomial of degree $n+1-i$ is $\Delta_{B|A}^{i}$. Thus
\[ P^n_{R|A} \cong \bigoplus_{0\leq i \leq n, |\alpha|=n-i} P^i_{B|A}[x_1,\dots,x_d] z^\alpha,
\]
\begin{samepage}and this proves the first statement. When $B$ is a field, each $P^i_{B|A}$ is free as a module over~$B$, and thus $P^i_{B|A}[x_1,\dots,x_d]$ is free over $R$. The local statement follows from the fact that 
$$(P^n_{R | A})_Q \cong P^n_{R_Q | A},$$
which is shown in \cite[Proposition 2.16]{DiffSig}.
\end{samepage}
\end{proof}

\

\section{The equal characteristic case} \label{Section field}

\subsection{Polynomial rings over a field}

\begin{lemma}\label{lemma injective}
Let $A$ be any ring and $(R,\m,k)$ be a local $A$-algebra. Let 
$$\overline{d} \in D^n_{R|A}\big(\, R/\m^{n+1}, \, k \otimes_R P^{n+1}_{R|A}\, \big)$$ 
be the operator induced from the universal operator $d \in D^n_{R|A}(R,P^{n}_{R|A})$. If the image of $A$ is contained in a coefficient field for $R/\m^{n}$, then $\overline{d}$ is injective.
\end{lemma}

\begin{proof}
First, we claim that $R/\m \otimes_R P^n_{R/\m^{n+1}|A} \cong  R/ \m \otimes_R P^n_{R|A}$.
Indeed, by \Cref{lemma powers moving around in tensor},
\[ R \otimes_A \m^{n+1} + \Delta^{n+1}_{R|A} \subseteq \m \otimes_A R + \Delta^{n+1}_{R|A},\]
and thus
\[ R/\m \otimes_R P^n_{R/\m^{n+1}|A} \cong \frac{R \otimes_A R}{\m \otimes R + R \otimes \m^{n+1} + \Delta^{n+1}_{R|A}} = \frac{R \otimes_A R}{\m \otimes R + \Delta^{n+1}_{R|A}} = R/ \m \otimes_R P^n_{R|A}.\]

We can now replace $R$ by $R/\m^{n+1}$ and assume that $\m^{n+1}=0$ without loss of generality. We assume that $R$ has a coefficient field $K$ containing the image of $A$; let $i\!: K \to R$ be the inclusion map. Let $\pi: R \to k$ be the quotient map, and $\sigma:k \to K$ be the inverse of $\pi \circ i$. Observe that $\sigma$ is an $A$-algebra map.
Let $\varepsilon$ denote the composition
\[\varepsilon\!: k\otimes_R (R \otimes_A R) \cong k \otimes_A R \xrightarrow{\, \sigma \otimes \mathbf{1} \,}  K \otimes_A R \xrightarrow{\, i \otimes \mathbf{1} \,} R \otimes_A R \xrightarrow{\mu_{R|A}} R.\]

We claim that 
$$\varepsilon(\Delta^{n+1}_{R|A} (k \otimes_A R) ) =0.$$ 
Since $\varepsilon$ is an $A$-algebra homomorphism, it suffices to show that $\varepsilon(\Delta_{R|A} (k \otimes_A R ) ) \subseteq \m$. The ideal $\Delta_{R|A}$ is generated by elements of the form $1\otimes r - r\otimes 1 \in R\otimes_A R$ for $r\in R$. Since $R = K \oplus \m$, we can write $r\in R$ as $\lambda + m$ with $\lambda\in K$ and $m\in \m$. Then, 
\[ \begin{aligned}\varepsilon(1\otimes r - r\otimes 1) 
&= \varepsilon(1\otimes (\lambda+m) - (\lambda+m)\otimes 1) \\ 
&=\varepsilon(1\otimes \lambda - \lambda \otimes 1 + 1\otimes m - m \otimes 1) \\
&=\varepsilon(1\otimes \lambda - \pi(\lambda) \otimes 1 + 1\otimes m) \\
&=\varepsilon(1\otimes \lambda - \pi(\lambda) \otimes 1 + 1\otimes m) \\
&=\lambda - i \circ \sigma \circ \pi (\lambda) + m.
\end{aligned} \]
But then 
\[ \pi(\lambda - i \circ \sigma \circ \pi (\lambda) + m) = \pi(\lambda) - \pi \circ i \circ \sigma \circ \pi (\lambda) + \pi(m) = \pi(\lambda) - \pi(\lambda) + 0 = 0,\]
so $\varepsilon(1\otimes r - r\otimes 1)\in \m$. This shows that 
$$\varepsilon(\Delta^{n+1}_{R|A} (k \otimes_A R) ) =0.$$ 
Thus, $\varepsilon$ descends to a well-defined $A$-algebra map
\[  k\otimes_R P^{n}_{R|A} \cong \frac{k \otimes_A R}{\Delta^{n+1}_{R|A} (k \otimes_A R)} \longrightarrow R,\]
which we also denote by $\varepsilon$.

Finally, we observe that $\varepsilon$ is a left inverse for $d$. Indeed, $d(r) = 1 \otimes r$ in $k \otimes_A P^{n}_{R|A}$, and $\varepsilon(1 \otimes r) = r$. This shows that $d$ is injective.
\end{proof}

\begin{proposition}\label{localize} 
Let $A$ be any ring, and let $K$ be a field which is an $A$-algebra. Let $R=K[x_1,\ldots,x_d]$, and consider prime ideals $Q$ and $\m$ of $R$ and a $Q$-primary ideal $\a$ such that $\a \subseteq Q \subseteq \m$. For all $n \geqslant 1$,
$$(\a\dif{n}{A})R_{\m} = (\a R_{\m})\dif{n}{A}.$$
\end{proposition}

\begin{proof}
By \Cref{module of principal parts direct sum}, $P^{n-1}_{R|A}$ is a free $R$-module. 
Consider the commutative diagram
\[
\xymatrix{
R \ar[d] \ar[r]^-{d} & P^{n-1}_{R|A} \ar[d] \\
R_\m \ar[r]^-{d'} & (P^{n-1}_{R|A})_\m \cong P^{n-1}_{R_\m|A}
}
\]
where the vertical maps are the natural localizations, the horizontal maps are the universal derivations, and the isomorphism on the right-bottom corner follows from \cite[Proposition 2.16]{DiffSig}. Given $a \in R$, since $P^{n-1}_{R|A}$ is free we have that 
$$a \notin \a \dif{n}{A} \quad \text{ if and only if } \quad d(a) \notin \a P^{n-1}_{R|A},$$ 
and this happens if and only if 
$$d' \left( \frac{a}{1} \right) = \frac{d(a)}{1} \notin \a P^{n-1}_{R_\m|A}.$$ 
Since $P^{n-1}_{R_\m|A}$ is also free, the latter happens if and only if 
$$\frac{a}{1} \notin (\a R_\m)\dif{n}{A}.$$ 
This shows that
$$a \in \a \dif{n}{A} \qquad \text{if and only if} \qquad \frac{a}{1} \in (\a R_\m)\dif{n}{A}.$$
The ideal $\a\dif{n}{A}$ is $Q$-primary \cite[Proposition 3.2]{ZNmixed}, so
\[
(\a\dif{n}{A})R_\m = \left\{\frac{a}{b} \ \bigg| \  a \in \a\dif{n}{A}, b \notin \m\right\}.
\]
On the other hand, given an element $\frac{a}{b} \in R_\m$, 
$$\frac{a}{b} \in (\a R_\m)\dif{n}{A} \quad \text{if and only if} \qquad \frac{a}{1} \in (\a R_\m)\dif{n}{A},$$
but we have just shown this is equivalent to $a \in \a\dif{n}{A}$. This completes the proof.
\end{proof}

\begin{theorem}\label{thm general field}
Let $K$ be any field and $R=K[x_1, \ldots, x_d]$. For all prime ideals $Q$ in $R$ and all $n \geqslant 1$, 
$$Q\dif{n}{\ZZ} = Q^{(n)}.$$ 	
\end{theorem}

\begin{proof}
By \cite[Proposition 2.6]{SurveySP}, the ideal $Q\dif{n}{\ZZ}$ is $Q$-primary, and by \cite[Proposition 2.5]{SurveySP}, $Q\dif{n}{\ZZ}$ contains $Q^n$. Note that \cite[Proposition 2.5]{SurveySP} and \cite[Proposition 2.6]{SurveySP} are stated for $Q\dif{n}{K}$, but the proofs do not use this assumption; see also \cite[Proposition 1.3]{Brumfiel} for an even more general statement. Moreover, $Q^{(n)}$ is by definition the smallest $Q$-primary ideal containing $Q^n$, and thus $Q^{(n)} \subseteq Q\dif{n}{\ZZ}$. So we only need to show that $Q\dif{n}{\ZZ} \subseteq Q^{(n)}$.

To prove it, we will show that the containment holds after localizing at the unique associated prime $Q$ of $Q^{(n)}$, which is $Q$. That is, we will show that $Q\dif{n}{\ZZ}R_Q \subseteq Q^{(n)}R_Q = (QR_Q)^n$. By \Cref{localize}, 
$$Q\dif{n}{\ZZ}R_Q = (QR_Q)\dif{n}{\ZZ}.$$
Replace $R$ by $R_Q$ and let $\m = QR_Q$. We have reduced the problem to showing that $\m\dif{n}{\ZZ} \subseteq \m^n$. Fix $f \notin \m^{n}$. Let $\pi \!: R \to R/\m^{n}$ be the canonical projection map, and note that $\pi(f) \neq 0$. By \Cref{lemma injective}, the map $\overline{d}\!: R/\m^{n} \to R/\m \otimes_R P^{n-1}_{R|\ZZ}$ induced by the universal differential operator $d\!: R \to P^n_{R|\ZZ}$  is injective, so $\overline{d}(\pi(f)) \neq 0$. The commutative diagram
$$\xymatrix{R \ar[r]^-d \ar[d]_-\pi & P^{n-1}_{R|\ZZ} \ar[d] \\ 
R/\m^{n} \ar[r]^-{\overline{d}} & R/\m \otimes_R P^{n-1}_{R|\ZZ}}$$
shows that $d(f)$ is a minimal generator of $P^{n-1}_{R|\ZZ}$. But $P^{n-1}_{R|\ZZ}$ is a free $R$-module by \Cref{module of principal parts direct sum}, and thus $d(f)$ generates a free summand of $P^{n-1}_{R|\ZZ}$. In particular, there exists an $R$-linear map $s\!: P^{n-1}_{R|\ZZ} \to R$ sending $d(f)$ to $1$. The composition $\partial \colonequals s \circ d$ is a $\ZZ$-linear differential operator of order up to $n$ on $R$, and $\partial(f) \notin \m$. Therefore, $f \notin \m\dif{n}{\ZZ}$.
\end{proof}

\begin{remark}\label{redu1}
    While \Cref{thm general field} is stated for prime ideals, the statement follows immediately for any radical ideal: it is elementary to check that taking differential powers commutes with taking intersections, so if $I$ has minimal primes
    \[ I = P_1 \cap \cdots \cap P_s,\]
    then
    \[\begin{aligned}
    I^{(n)} & = P_1^{(n)} \cap \cdots \cap P_s^{(n)} \\
    & = P_1\dif{n}{\ZZ} \cap \cdots \cap P_s\dif{n}{\ZZ} & \text{by \Cref{thm general field}}\\
    & = \left( P_1 \cap \cdots \cap P_s \right)\dif{n}{\ZZ} = I\dif{n}{\ZZ}.
    \end{aligned} \]
\end{remark}

One of the difficulties of the proof of \Cref{thm general field} is that, in general, the modules $P^n_{R|\ZZ}$ are far from being finitely generated. Our next goal is to show that, given a fixed prime $Q$ in $R=K[x_1,\ldots,x_n]$, one can find a suitable subfield $K_0 \subseteq K$ such that all modules of principal parts $P^n_{R|K_0}$ are finitely generated over $R$, and such that symbolic powers of $Q$ can be computed using $K_0$-linear differential operators.

\begin{lemma}\label{lemma field 1} 
Let $K$ be a field of characteristic $p>0$. Let $R$ be a local ring essentially of finite type over $K$ with residue field $k$. There exists a field $K_0\subseteq K$ such that the extension has a finite $p$-basis and $K_0$ is contained in a coefficient field for $\widehat{R}$.
\end{lemma}

\begin{proof} 
Let $\Theta$ be a $p$-basis for $K$. Since $k$ is a finitely generated field extension of $K$, the vector spaces $\Omega_{k|K}$ and $\Gamma_{k|K}$ are finite dimensional, with $\Gamma_{k|K}=\Gamma_{k|K|\FF_p}$ as in \cite[\S26]{Matsumura}. Fix a $p$-basis $\tau$ for $k/K$, so $\{ dt \mid t\in \tau\}$ is a basis for $\Omega_{k|K}$. From the Jacobi-Zariski sequence
\[ 0 \to \Gamma_{k|K} \to k \otimes_K \Omega_{K|\FF_p} \to \Omega_{k|\FF_p} \to \Omega_{k|K} \to 0,\]
since the first and last modules are finite dimensional by \cite[Theorem~26.10]{Matsumura}, there is a finite subset $\Theta'' \subseteq \Theta$ with $\Theta'=\Theta \smallsetminus \Theta''$ such that $\{ d\theta \ | \ \theta\in \Theta'\} \cup \{ dt \ |\ t\in \tau\}$ is a basis for $\Omega_{k|\FF_p}$, and hence $\Theta' \cup \tau$ is a $p$-basis for $k$. Let $\tilde{\tau}$ be a set of lifts of $\tau$ to $\widehat{R}$. Then, by the classical construction of coefficient fields in positive characteristic \cite{Cohen}, $K' \colonequals \bigcap_e R^{p^e}[\Theta' \cup \tilde{\tau}]$ is the unique coefficient field for $\hat{R}$ containing $\tilde{\tau}$. Set 
$$K_0 \colonequals \bigcap_e K^{p^e}[\Theta'] = \bigcap_e K^{p^e}(\Theta').$$ 
By construction, $K_0 \subseteq K'$ and $K_0 \subseteq K$. Moreover, since $K_0$ is a nested intersection of fields, it is also a field. We claim that $\Theta''$ contains a $p$-basis for $K/K_0$. Indeed, 
\[  K^p(K_0)(\Theta'') = K^p \left( \bigcap K^{p^e}(\Theta') \right)(\Theta'') \supseteq K^p(\Theta',\Theta'') = K. \qedhere\]
\end{proof}

\begin{lemma}\label{lemma field 2} 
Let $R$ be a ring essentially of finite type over a field $K$ of characteristic $p>0$. If $K_0\subseteq K$ is a subfield such that $K/K_0$ has a finite $p$-basis, then $P^n_{R|K_0}$ is finitely generated for all $n$.
\end{lemma}

\begin{proof} 
Without loss of generality, we can assume $R$ is algebra-finite over $K$.
Fix $n$, and take some $e$ such that $p^e>n$. Consider the ring inclusions $K_0 \subseteq K_0[R^{p^e}] \subseteq R$. Since 
\[R\otimes_{K_0[R^{p^e}]} R \cong (R\otimes_{K_0} R)/(\Delta_{K_0[R^{p^e}]|K_0})(R\otimes_{K_0} R),\] 
we have 
\[P^n_{R|K_0[R^{p^e]}} \cong P^n_{R|K_0} / \Delta_{K_0[R^{p^e}]|K_0} P^n_{R|K_0}.\] 
But 
\[ \Delta_{K_0[R^{p^e}]|K_0} \subseteq \Delta_{R|K_0}^{[p^e]} \subseteq \Delta_{R|K_0}^{p^e} \subseteq \Delta_{R|K_0}^{n},\] 
so we have 
\[ P^n_{R|K_0} = P^n_{R|K_0[R^{p^e}]}.\] 
Since $K^{p^e} \subseteq R^{p^e}$, there is a surjective map of $R$-modules
\[ P^n_{R|K_0(K^{p^e})} \longrightarrow P^n_{R|K_0[R^{p^e}]}\] 
and thus it suffices to show that $P^n_{R|K_0(K^{p^e})}$ is a finitely generated $R$-module. But the $p$-basis of $K/K_0$ generates $K$ as an algebra over $K_0(K^{p^e})$, so $R$ is algebra-finite over $K_0(K^{p^e})$, and hence the module of principal parts is finitely generated.
\end{proof}

\begin{theorem}
Let $K$ be any field and $R=K[x_1,\dots, x_d]$. For every prime ideal $Q$ in $R$, there exists a subfield $K_0 \subseteq K$ such that for all $n \geqslant 1$, 
$$Q\dif{n}{K_0} = Q^{(n)}$$
and the modules of differential operators $D^n_{R|K_0}$ and principal parts $P^n_{R|K_0}$ are finitely generated over $R$.
\end{theorem}

\begin{proof}
First, note that $Q^{(n)} = Q\dif{n}{\ZZ} \subseteq Q\dif{n}{L}$ for any field $L$. Let $S=R_Q$, and let $K_0 \subseteq K$ be as in \Cref{lemma field 1} applied to $S=R_Q$. By \Cref{lemma field 2} applied to the extension $K/K_0$ and the ring $R$, and \Cref{module of principal parts direct sum}, we conclude that $P^{n-1}_{R|K_0}$ is a finitely generated free $R$-module. 

Let $\m = QR_Q$ be the maximal ideal of $S$. Since $K_0$ is contained in a coefficient field of~$\widehat{S}$, it is contained in a coefficient field of $S/\m^{n}$. In particular, 
$$\overline{d}: S/\m^{n} \to P^{n-1}_{S|K_0} \otimes_S S/\m$$ 
is injective by \Cref{lemma injective}. If $f \notin Q^{(n)}$, then its image $f/1$ in $S/\m^n$ is nonzero, and by injectivity $\overline{d} \left( f/1 \right) \neq 0$. Since $P^{n-1}_{S|K_0}$ is a free $S$-module, the element $\overline{d}(f/1)$ must generate a free $S$-summand. We can then find $\varphi \in \Hom_S(P^{n-1}_{S|K_0},S)$ such that $\varphi(\overline{d}(f/1)) = 1$, and since $\varphi \circ \overline{d}$ is a $K_0$-linear differential operator of order up to $n$, we conclude that $f/1 \notin \m\dif{n}{K_0}$. Using \Cref{localize}, 
$$\frac{f}{1} \notin \m\dif{n}{K_0} = (QR_Q)\dif{n}{K_0} \cong \left(Q\dif{n}{K_0}\right)R_Q.$$ 
We conclude that $f \notin  Q\dif{n}{K_0}$.
\end{proof}

While the field $K_0$ in the previous theorem works for $Q^{(n)}$ for every $n$, it cannot be chosen independently of $Q$, as the next example shows.

\begin{example} 
Consider $K=\FF_p(t_1,t_2,\dots)$ and $R=K[x]$. Let $K_0$ be a subfield of $K$. Note that $P^1_{R|K_0} = \Omega_{R|K_0} \oplus R$ is finitely generated if and only if $K/K_0$ admits a finite $p$-basis.
In particular, if $P^1_{R|K_0}$ is finitely generated then there is some $i$ such that $t_i$ is not part of a $p$-basis for $K/K_0$. Take $Q=(x^p-t_i)$. Then $d t_i =0$ in $\Omega_{K|K_0}$, and hence in $\Omega_{R|K_0}$. Thus every differential operator of order $1$ kills $t_i$. Likewise, every derivation kills $x^p$, so every derivation kills $x^p-t_i$. Thus
\[ x^p-t_i \in Q\dif{2}{K_0} \smallsetminus Q^{(2)}.\]
\end{example}

\subsection{Singular $K$-algebras}

Let $K$ be a field, $R$ a $K$-algebra essentially of finite type over~$K$, and $Q$ is a prime ideal of $R$.
In \cite[Theorem~C]{Yairon}, Cid-Ruiz proves that if the extension $K \hookrightarrow R_Q/QR_Q$ is separable, then
\[
Q^{(n)} = \{f \in R \mid \partial(f) =0 \text{ for all } \partial \in D^{n-1}_{R|K}(R,R/Q)\}.
\]

We remove the separability assumption on the field extension.

\begin{theorem} \label{thm Yairon}
Let $K$ be a field, and $R$ be a $K$-algebra essentially of finite type over $K$. For all prime ideals $Q$ of $R$,
\[
Q^{(n)} = \{f \in R \mid \partial(f) = 0 \text{ for all } \partial \in D^{n-1}_{R|\ZZ}(R,R/Q)\}.
\]
Moreover, for each prime ideal $Q$, there exists a subfield $K_0 \subseteq K$ such that for all $n \geqslant 1$,
\[
Q^{(n)} = \{f \in R \mid \partial(f) = 0 \text{ for all } \partial \in D^{n-1}_{R|K_0}(R,R/Q)\},
\] 
and the modules of differential operators $D^n_{R|K_0}$ and principal parts $P^n_{R|K_0}$ are finitely generated over $R$.
\end{theorem}

\begin{proof}
Fix a prime $Q$. Let $S=R_Q$ with maximal ideal $\m=QS_Q$. By \Cref{lemma field 1} applied to $S$, there exists a field $K_0 \subseteq K$ such that the extension has a finite $p$-basis, and $K_0$ is contained in a coefficient field for $\widehat{S}$. 

Set
\[ Q\cidr{n} \colonequals  \{f \in R \mid \partial(f) = 0 \text{ for all } \partial \in D^{n-1}_{R|\ZZ}(R,R/Q) \}\]
and 
\[ Q^{\left\{n\right\}_{K_0}} \colonequals  \{f \in R \mid \partial(f) = 0 \text{ for all } \partial \in D^{n-1}_{R|K_0}(R,R/Q) \}.\]
By \cite[Proposition 1.3]{Brumfiel}, the ideals $Q\cidr{n}$ and $Q^{{\left\{n\right\}}_{K_0}}$ are both $Q$-primary and contain $Q^n$, and hence they must both contain $Q^{(n)}$. Note that any $K_0$-linear differential operator is in particular $\ZZ$-linear, and thus we must have
\[ Q^{(n)} \subseteq Q\cidr{n} \subseteq Q^{\left\{n\right\}_{K_0}}.\]

Now consider $f \notin Q^{(n)}$, and note that the image $f/1$ of $f$ in $S = R_Q$ does not belong to $\m^n$. 
Moreover, by \Cref{lemma field 2} applied to $K_0 \subseteq K$ and $R$, the $R$-module $P^{n-1}_{R|K_0}$ is finitely generated. By \Cref{lemma injective} we have that $\overline{d}(f/1) \notin \m P^{n-1}_{S|K_0}$ and, in particular, we can find a map $\varphi: P^{n-1}_{S|K_0} \to S/\m = \kappa(Q)$ such that $\varphi(\overline{d}(f/1)) \neq 0$.

Since $P^{n-1}_{R|K_0}$ is finitely generated and $(P^{n-1}_{R|K_0})_Q \cong P^{n-1}_{S|K_0}$, we have an isomorphism
\[ 
\Hom_{R_Q}(P^{n-1}_{R_Q|K_0},\kappa(Q)) \cong \left(\Hom_R(P^{n-1}_{R|K_0},R/Q)\right)_Q.
\] 
In particular, there exists $\psi\!: P^{n-1}_{R|K_0} \to R/Q$ such that $\psi(d(f)) \ne 0$. Since 
\[
\partial = \psi \circ d \in D^{n-1}_{R|K_0}(R,R/Q) \subseteq D^{n-1}_{R|\ZZ}(R,R/Q),
\] 
we have found a $K_0$-linear (and thus $\ZZ$-linear) differential operator of order up to $n-1$ that sends $f$ outside of $Q$. We conclude that $f \notin Q^{{\left\{n\right\}}_{K_0}}$. This shows that
\[ Q^{{\left\{n\right\}}_{K_0}} \subseteq Q^{(n)}\]
and since we showed above that $Q^{(n)} \subseteq Q\cidr{n} \subseteq Q^{{\left\{n\right\}}_{K_0}}$, we conclude that
\[ Q\cidr{n} = Q^{{\left\{n\right\}}_{K_0}} = Q^{(n)}. \qedhere\]
\end{proof}

\begin{remark}\label{redu2}
One can write \Cref{thm Yairon} for any radical ideal.
If
\[ I = Q_1 \cap \cdots \cap Q_s,\]
then 
\[\begin{aligned}
I^{(n)} & = Q_1^{(n)} \cap \cdots \cap Q_s^{(n)} \\
& = \{f \in R \mid \partial(f) = 0 \text{ for all } \partial \in D^{n-1}_{R|\ZZ}(R,R/Q_i) \text{ and all } i \}  
& \text{by \Cref{thm Yairon}}.
    \end{aligned} \]
The functor $D^n_{R|\ZZ}(R, -)$ is left exact, so from the exact sequence
\[ \xymatrix{0 \ar[r] & R/I \ar[r] & \displaystyle\bigoplus_{i=1}^s R/Q_i}\]
we conclude that every differential operator in $D^n_{R|\ZZ}(R,R/I)$ embeds in $\bigoplus\limits_{i=1}^s D^n_{R|\ZZ}(R,R/Q_i)$.

Thus if $D^{n-1}_{R|\ZZ}(R,R/Q_i)$ sends $f$ to $0$ for all $i$, then every $D^{n-1}_{R|\ZZ}(R,R/I)$ must send $f$ to $0$ as well. Conversely, if $f \notin I^{(n)}$, then $f \notin Q_i^{(n)}$ for some $i$, and thus by \Cref{thm Yairon} there exists a differential operator $\partial \in D^{n-1}_{R|\ZZ}(R,R/Q_i)$ such that $\partial(f) \neq 0$. By prime avoidance, one can find 
\[ v \in \left(\bigcap_{j \neq i} Q_j\right) \smallsetminus Q_i, \]
and the composition
\[ \xymatrix{\alpha\!: R \ar[r]^-{\partial} & R/Q_i \ar[r]^-{\cdot v} & R/I}\]
is a differential operator $\alpha \in D^{n-1}_{R|\ZZ}(R,R/I)$ such that $\alpha(f) \neq 0$. Thus
\[ I^{(n)} = \{f \in R \mid \partial(f) = 0 \text{ for all } \partial \in D^{n-1}_{R|\ZZ}(R,R/I) \}.\]
See \cite{YaironStrumfels} for characterizations of nonprimary ideals via differential operators.
\end{remark}

\

\section{The mixed characteristic case}

We start by recalling the notion of $p$-derivation, introduced independently by Joyal \cite{Joyal} and Buium \cite{Buium1995}.

\begin{definition}\label{p-derivation definition}
Let $S$ be a ring, and $p$ a prime integer that is a nonzerodivisor on $S$. A set-theoretic function $\delta\!: S \to S$ is a {\bf $p$-derivation} if the function $\phi_p\!:S \to S$ defined by $\phi_p(x) = x^p+p\delta(x)$ is a ring homomorphism.
\end{definition}

By rewriting the definition of a ring homomorphism in terms of $\delta$, one can show that $\delta$ is a $p$-derivation if and only if it satisfies the following properties:
\begin{itemize}
	\item $\delta(1) = 0$.
	\item $\delta(xy) = x^p\delta(y) + \delta(x)y^p + p\delta(x)\delta(y)$ for all $x, y \in S$.
	\item $\delta(x+y) = \delta(x) + \delta(y) + \frac{x^p+y^p-(x+y)^p}{p}$ for all $x,y \in S$.
\end{itemize}

Now let $V$ be a discrete valuation ring with uniformizer $p>0$, and assume that the polynomial ring $R=V[x_1,\ldots,x_m]$ has a $p$-derivation $\delta$. In our previous work in \cite{ZNmixed}, we introduced the notion of mixed differential powers in order to characterize symbolic powers of prime ideals $Q$ of $R$ that contain $p>0$: 
\[ Q\difMV{n} \colonequals \{f \in R \mid (\delta^a \circ \partial)(f) \in Q \text{ for all } \partial \in D_{R|V}^b \text{ with } a+b \leqslant n-1\}.\]
We proved in the main result of \cite{ZNmixed} that if $p \in Q$ and the extension $V/(p) \hookrightarrow R_Q/QR_Q$ is separable then $Q^{(n)} = Q\difMV{n}$. The separability assumption was not only needed in the proof, but crucial for the validity of the statement. We recall an illustrative example to show why this is the case.

\begin{example}[{see \cite[Example~3.28]{ZNmixed}}]
Let $V=\ZZ[t]_{(p)}$ and $R=V[x]$. For the prime ideal $Q=(p,x^p-t)$, the extension $V/(p) \cong \FF_p(t) \hookrightarrow R_Q/QR_Q \cong \FF_p(t^{1/p})$ is not separable. In this case, one can show that $D^1_{R|V}(x^p-t) \subseteq Q$ and there is a $p$-derivation $\delta$ on $R$ such that $\delta(x^p-t) \in (p) \subseteq Q$. It follows that $x^p-t \in Q\difMV{2} \smallsetminus Q^{(2)} = Q^2$.
\end{example}

The goal of this section is to remove the assumption of separability, and to do so we must slightly change the definition of mixed differential powers given in \cite{ZNmixed} to include all differential operators.

\begin{definition}
Let $A$ be a ring, and $S$ be an $A$-algebra with a $p$-derivation $\delta$. Given an integer $n$ and a prime ideal $Q$ of $S$, let
\[ 
Q\derp{s} \colonequals \{f \in S \mid \delta^a(f) \in Q  \text{ for all } a \leqslant s-1\}.
\]
We define the $n$th mixed differential power of $Q$ as
\begin{align*}
Q\difM{n} & \colonequals \{f \in S \mid (\delta^a \circ \partial)(f) \in Q \text{ for all } \partial \in D_{S|\ZZ}^b \text{ with } a+b \leqslant n-1\} \\
& = \bigcap_{s+t \leq n+1} \left(Q\derp{s}\right)\dif{t}{\ZZ}.
\end{align*}
\end{definition}

The mixed differential power $Q\difM{n}$ of a prime ideal $Q$ is a $Q$-primary ideal of $R$ that contains $Q^n$, and thus $Q\difM{n} \supseteq Q^{(n)}$ for all $n \geqslant 1$ \cite[Proposition 3.19]{ZNmixed}. 

\begin{remark} 
A priori the definition depends on the fixed $p$-derivation $\delta$. We will show in our main theorem, \Cref{main}, that the definition is actually independent of this choice.
\end{remark}

Before stating the main theorem of this section, we need some preliminary results.

\begin{lemma}\label{d injective}
Let $V$ be a complete discrete valuation ring with uniformizer $p>0$, and let $Q$ be a prime ideal in $S=V[x_1,\ldots,x_n]$ that contains $p$. Let $(R,\m) = (S_Q,QS_Q)$, and let $f_1,\ldots,f_v,p$ be a minimal generating set of $\m$. Set $\n=(f_1,\ldots,f_v)$. Let $t>0$ and consider the universal derivation $d\!: R \to P^{t-1}_{R|\ZZ}$. The induced map 
\[ \overline{d}:R/(p, \n^t) \longrightarrow P^{t-1}_{R|\ZZ}/\m P^{t-1}_{R|\ZZ} \]
is injective.
\end{lemma}

\begin{proof}
To simplify notation, set $P=P^{t-1}_{R|\ZZ}$ and $T=R/(p)$. Let $\overline{\Delta}$ denote the image of the diagonal $\Delta = \ker(R \otimes_\ZZ R \to R)$ in $T/\n \otimes_\ZZ T/\n^t$.
First, we claim that
\[
P/\m P \cong (T/\n \otimes_\ZZ T/\n^t)/\overline{\Delta}^t.
\] 
Indeed,
$$\begin{aligned}
\frac{P}{\m P} & \cong \frac{T \otimes_A T}{\Delta^t} \Bigg/  \frac{\n \otimes T}{\Delta^t} \vspace{0.5em} \\
& \cong \frac{T \otimes_A T}{\n \otimes T + {\Delta}^t} \\
& = \frac{T \otimes_A T}{\n \otimes T + T \otimes \n^t + {\Delta}^t} \vspace{1.5em} & \quad \textrm{since } \n \otimes T + \Delta^t \supseteq T \otimes \n^t \textrm{ by \Cref{lemma powers moving around in tensor}} \\
& \cong \left(\frac{T}{\n} \otimes \frac{T}{\n^t}\right) \!\bigg/ \overline{\Delta}^t.
\end{aligned}$$

Since $T/\n^t$ is Artinian of characteristic $p$, it has a coefficient field $T_0 \subseteq T/\n^t$. 
Consider the composition $\varepsilon$
\[
\xymatrix{
T/\n \otimes_\ZZ T/\n^t \ar@/_2.0pc/[rrr]_-{\varepsilon} \ar[r]^-{\cong} & T_0 \otimes_\ZZ T/\n^t \, \ar[r] & T/\n^t \otimes_\ZZ T/\n^t \ar[r]^-{\mu} & T/\n^t.
}
\]
We now claim that $\varepsilon(\overline{\Delta}) \subseteq \n (T/\n^t)$. Since $\overline{\Delta} \subseteq T/\n \otimes_\ZZ T/\n^t$ is the ideal generated by the residue classes $\overline{r \otimes 1 - 1 \otimes r}$ of $r \in T/\n^t$ inside $T/\n \otimes_\ZZ T/\n^t$, it suffices to show that 
\[ \varepsilon(\overline{r\otimes 1 -1\otimes r}) \in \n(T/\n^t).\] 
Recall that any $r \in T/\n^t$ can be written as $r=t_0+\eta$ for some $t_0 \in T_0$ and $\eta \in \n (T/\n^t)$. Thus in $T/\n \otimes_\ZZ T/\n^t$ we have
\[ 
\overline{r \otimes 1 - 1 \otimes r}  = \overline{t_0 \otimes 1 - 1 \otimes t_0 - 1 \otimes \eta}. \] 
Moreover, $t_0 \otimes 1 - 1 \otimes t_0$ is killed by the multiplication map when restricted to $T_0 \otimes_\ZZ T_0$, and hence is in the kernel of $\varepsilon$. Thus
$$\begin{aligned}
\varepsilon(\overline{r \otimes 1 - 1 \otimes r} ) & = \varepsilon(\overline{t_0 \otimes_1 - 1 \otimes t_0 - 1 \otimes \eta}) \\
    & = \varepsilon(t_0 \otimes 1 - 1 \otimes t_0) - \varepsilon(1 \otimes \eta) \\
    & = -\varepsilon(1 \otimes \eta) \\
    & = -\eta \in \n(T/\n^t).
\end{aligned}$$
This proves the claim.
Note that the fact that $\varepsilon(1 \otimes \eta) = \eta$ follows from the definition of $\varepsilon$.

Now that we have proved the claim, since $\varepsilon$ is a ring homomorphism, we have
\[ \varepsilon(\overline{\Delta}^t) \subseteq [\n(T/\n^t)]^t = 0.\]
In particular, we have an induced map 
\[
\widetilde{\varepsilon}: P/\m P \cong (T/\n \otimes_\ZZ T/\n^t)/\overline{\Delta} \longrightarrow T/\n^t = R/(p,\n^t),
\]
and $\widetilde{\varepsilon} \circ \overline{d}(r) = \widetilde{\varepsilon}(\overline{1\otimes r})= r$. This concludes the proof that $\overline{d}$ is injective.
\end{proof}

\begin{lemma} \label{lemma map}
Let $(V, pV, K)$ be a complete discrete valuation ring with uniformizer $p>0$, and let $Q$ be a prime ideal of $ S=V[x_1,\ldots,x_n]$ that contains $p$. Let $t>0$ and $x \in P^{t-1}_{S|\ZZ}$. If $x \notin QP^{t-1}_{S|\ZZ}$, then there exists an $S$-linear map $\gamma\!:P^{t-1}_{S|\ZZ} \to S$ such that $\gamma(x) \notin Q$.
\end{lemma}

\begin{proof}
By \Cref{module of principal parts direct sum} and \cite[Proposition 2.16]{DiffSig}, 
$$P^{t-1}_{S|\ZZ} \cong  \bigoplus_{\lambda \in \Lambda} \left( P^i_{V|\ZZ} \otimes_V S \right),$$
where $\Lambda = \{\lambda \in \NN^m \mid |\lambda|+i=t-1\}$.
Since $x = (x_\lambda)_{\lambda \in \Lambda} \notin Q P^{t-1}_{S|\ZZ}$, there exists a component $x_{\lambda_0} \in P_0 = P^{i_0}_{V|\ZZ} \otimes_V S$ of the direct sum such that $x_{\lambda_0} \notin Q P_0$. Let $M = P^{i_0}_{V|\ZZ}$, and write 
\[ x_{\lambda_0} = \sum_{j=1}^u \alpha_j \otimes s_j \qquad \text{for some } \alpha_j \in M \text{ and } s_j \in S. \] 
Without loss of generality, after possibly rearranging and rewriting the sum, we may assume that $\alpha_1,\ldots,\alpha_r$ are linearly independent mod $p M$, and that $\alpha_{r+1},\ldots,\alpha_u \in pM$. Moreover, our running assumptions guarantee that there exists $a=1,\ldots,r$ such that $s_a \notin Q$. Without loss of generality we may assume $s_1 \notin Q$.

Consider the map \vspace{-0.5em}
$$\begin{aligned} 
\psi\!: V^r & \longrightarrow M \\
e_i & \longmapsto \alpha_i .
\end{aligned}$$
We claim that $\psi \otimes \text{id} \!: V^r \otimes_V V/(p^j) \longrightarrow M \otimes_V V/(p^j)$ is injective for all $j>0$. We will use induction on $j$. The case $j=1$ follows from the discussion above. For the inductive step,
since $V$ is flat over $\ZZ$, by \cite[proof of Lemma 2.1]{BBLSZ1} we have $$M/(p^j) M \cong M \otimes_{\ZZ} \ZZ/(p^j) \cong P^{i_0}_{V/(p^j) |\ZZ/(p^j)}.$$
By \cite[Tag~013L]{stacks-project}, the map $\ZZ/(p^j) \to V/(p^j)$ is formally smooth. By \cite[16.10.2]{EGA-IV}, the $V/(p^j)$-module $P^{i_0}_{V/(p^j)|\ZZ/(p^j)}$ is projective and thus free, since $V/(p^j)$ is local. Thus, given 
\[ \sum_i v_i \alpha_i\in p^j M \subseteq p M,\] 
we can write $v_i = pv'_i$ for each $i$, and applying freeness, 
\[ \sum_i v'_i\alpha_i\in p^{j-1}M.\] 
By the induction hypothesis, we must have $v'_i\in p^{j-1}V$, so $v_i\in p^jV$, completing the proof that $\psi \otimes_V V/(p^j)$ is injective for all $j$.

We claim that $\psi$ is a split injection, so $V^r$ is a direct summand of $M$. To show this, let 
$$E \colonequals E_V(K)\cong V_{p}/V\cong \varinjlim V/p^j V$$ 
be the injective hull of the residue field~$K$ of $V$. The map $\psi\otimes_V E$ must be injective. Thus applying $\Hom_V(-,E^r)$ to the exact sequence
\[\xymatrix{0 \ar[r] & V^r \otimes_V E \ar[r]^-{\psi} & M \otimes_V E}\]
yields the exact sequence
\[\xymatrix{\Hom_V(M \otimes_V E, E^r) \ar[r]^-{\varphi} & \Hom_V(V^r \otimes_V E, E^r) \ar[r] & 0}\]
where $\varphi$ is given by precomposition with $\psi \otimes \text{id}$. Moreover, we have natural isomorphisms
\[\begin{aligned}
\Hom_V( - \otimes_V E, E^r) & \cong \Hom_V (-, \Hom_V(E,E^r)) & \text{by Hom-tensor adjunction}\\
& \cong  \Hom_V (-, \Hom_V(E,E)^r) \\
& \cong \Hom_V( -, \widehat{V}^r) \\
& = \Hom_V( -, V^r) & \text{since $V$ is complete}.
\end{aligned}\]
Therefore, we get an exact sequence
\[\xymatrix{\Hom_V(M,V^r) \ar[r]^-{\tilde\varphi} & \Hom_V(V^r,V^r) \ar[r] & 0}\]
where $\tilde\varphi$ is given by precomposition with $\psi$. We conclude that $\psi$ is a split injection. 

Thus there exists a $V$-linear map $\beta\!:M \to V$ such that $\beta(\alpha_1) = 1$ and 
\[\beta(\alpha_2) = \cdots = \beta(\alpha_r)=0.\]
By base change, this induces an $S$-linear map 
\[\gamma\!:P_0 = M \otimes_V S \to S\]
such that $\gamma(\alpha_1 \otimes s_1) = s_1$, while $\gamma(\alpha_a \otimes s_a) = 0$ for all $a=2,\ldots,r$. On the other hand, for all $a=r+1,\ldots,u$ we have $\alpha_a \in pM$, and since $p \in Q$, we must have
\[ \alpha_a \otimes s_a \in Q P_0.\] 
Therefore, for all $a=r+1,\ldots,u$ we have that $\gamma(\alpha_a \otimes s_a) \in Q$. It follows that 
\[\gamma(x_{\lambda_0}) \equiv s_1 \bmod Q,\] 
and since $s_1 \notin Q$, this shows that $\gamma(x_{\lambda_0}) \notin Q$. Extending by zero to the remaining components of $P^{t-1}_{S | \ZZ}$, this gives an $S$-linear map sending $x$ outside of $Q$.
\end{proof}

\begin{corollary} \label{existence diff operator} Let $V$ be a complete discrete valuation ring with uniformizer $p>0$, and let $Q$ be a prime ideal in $R=V[x_1,\ldots,x_n]$ that contains $p$. Let $t \geqslant 1$, and let $a \in R$ be an element such that its image in $R_Q$ does not belong to $(p,\n^t)R_Q$. Then there exists a differential operator $\partial \in D^{t-1}_{R|\ZZ}$ such that $\partial(a) \notin Q$.
\end{corollary}

\begin{proof}
Since the image of $a$ in $R_Q$ does not belong to $(p,\n^t)R_Q$, by \Cref{d injective} the image of $d(a)$ in $(P^{t-1}_{R|\ZZ})_Q \cong P^{t-1}_{R_Q|\ZZ}$ does not belong to $Q P^{t-1}_{R_Q|\ZZ}$. Therefore, $d(a) \notin QP^{t-1}_{R|\ZZ}$. By \Cref{lemma map}, we can then find an $R$-linear map $\varphi\!: P^{t-1}_{R|\ZZ} \to R$ such that $\varphi(d(a)) \notin Q$. Then $\partial = \varphi \circ d \in D^{t-1}_{R|\ZZ}$ is the desired differential operator.
\end{proof}

\begin{theorem} \label{main}
Let $V$ be a complete discrete valuation ring with uniformizer $p>0$, and $R=V[x_1,\ldots,x_n]$. Let $\delta$ be a $p$-derivation on $R$. If $Q$ is a prime ideal of $R$ that contains $p$, then $Q^{(n)} = Q\difM{n}$.
\end{theorem}

\begin{proof}
Recall that $Q\difM{n}$ is a $Q$-primary ideal containing $Q^n$, and thus containing $Q^{(n)}$. Assume by way of contradiction that $Q\difM{n} \not\subseteq Q^{(n)}$. Choose $f \in Q\difM{n} \smallsetminus Q^{(n)}$ with highest $p$-adic order, say of $p$-adic order $s$. Note that necessarily $s<n$, since $p^n \in Q^{(n)}$. We can write $f=p^sg+p^{s+1}h$ for some $g \notin (p)$. First, we claim that $p^sg \notin Q^{(n)}$. In fact, if $p^sg \in Q^{(n)}$, then necessarily $p^{s+1}h \notin Q^{(n)}$. However, $p^{s+1}h = f-p^sg \in Q\difM{n}$, so that $p^{s+1}h$ would be an element of $Q\difM{n} \smallsetminus Q^{(n)}$ of higher $p$-adic order than $f$, which contradicts our choice of~$f$.

Now let $S=R_Q$, and $\m=Q R_Q$ be its maximal ideal. Let $f_1,\ldots,f_v,\frac{p}{1}$ be a minimal generating set of $\m$, and let $\n = (f_1,\ldots,f_v)$. Set $t=n-s$. We claim that ${\frac{g}{1} \notin (\frac{p}{1},\n^t)}$. Otherwise, there would exist ${\alpha \notin Q}$ such that ${\alpha g \in (p)+Q^t}$, and thus 
\[ {\alpha p^sg \in (p^{s+1})+Q^{(n)}}.\] 
But then for some $u \in R$ we have
\[\alpha f = \alpha p^s g + \alpha p^{s+1}h = p^{s+1}u \quad \text{ modulo } Q^{(n)} .\] 
Note that ${\alpha f \in Q\difM{n} \smallsetminus Q^{(n)}}$ since $f \notin Q^{(n)}$, ${\alpha \notin Q}$, and $Q^{(n)}$ is $Q$-primary. But since $Q^{(n)} \subseteq Q\difM{n}$, it would then follow that $p^{s+1}u \in Q\difM{n} \smallsetminus Q^{(n)}$ as well, contradicting once again the original choice of $f$. This proves that ${\frac{g}{1} \notin (\frac{p}{1},\n^t)}$, as claimed.

Thus, we have 
\[ f=p^sg + p^{s+1}h \in Q\difM{n} \smallsetminus Q^{(n)} \] 
of highest possible $p$-adic order $s<n$, with $p^sg \notin Q^{(n)}$ and $\frac{g}{1} \notin (\frac{p}{1},\n^t)$ for $t=n-s$. By \Cref{existence diff operator}, we can find $\partial \in D^{t-1}_{S|\ZZ}$ such that $\partial(g) \notin Q$. Since $f \in Q\difM{n}$, then $\partial(f) \in Q\derp{s}$. Given that $p^{s+1} \in Q\derp{s}$, and that $Q\derp{s}$ is an ideal, we conclude that 
\[p^s \partial(g) = \partial(f) - p^{s+1}\partial(h) \in Q\derp{s}.\] 
However, $p^s \notin Q\derp{s}$, and since $Q\derp{s}$ is $Q$-primary, we therefore conclude that $\partial(g) \in Q$, contradicting our choice of $\partial$. Therefore we must have $f \notin Q\difM{n}$, as desired.
\end{proof}

\begin{remark}
    \Cref{main} easily extends to any radical ideal along the lines of \Cref{redu1} and \Cref{redu2}.
\end{remark}

\begin{example}\label{example Yairon mixed char}
One might ask whether, for a singular ring $R$ of mixed characteristic $(0,p)$ with a $p$-derivation $\delta$, one can characterize symbolic powers of $Q\ni p$ via compositions of $\delta$ and differential operators $D_{R|\ZZ}(R,R/Q)$, $D_{R|\ZZ}(R,R)$, and $D_{(R/Q)|\ZZ}(R/Q,R/Q)$ of the appropriate orders, in analogy with \cite{Yairon} and \Cref{thm Yairon}. However, this is not possible.

Take $R=\ZZ_{(p)}[x,y]/(xy)$, and $Q=(p,x,y)$. Note that $R$ admits a $p$-derivation $\delta$. For instance, the $p$-derivation on $\ZZ_{(p)}[x,y]$ mapping $x$ and $y$ to $0$ descends to $R$. We claim that every composition $\alpha \delta^b \gamma$ with 
\[\alpha\in D^a_{R|\ZZ}(R,R/Q), \quad \gamma\in D^c_{R|\ZZ}(R,R), \quad \text{and} \quad  a+b+c<3\] 
maps $px$ to an element of $Q$. Since $px\notin Q^{(3)}$ and $D_{(R/Q)|\ZZ}(R/Q,R/Q)$ is trivial, this will show our claim that symbolic powers cannot be characterized as described above. 

Note first that $(x,y)$ is a {$D_{R|\ZZ}(R,R)$-stable} ideal, meaning that $D_{R|\ZZ}(R,R) (x,y) \subseteq (x,y)$: indeed, both ideals $(x)$ and $(y)$ are {$D_{R|\ZZ}(R,R)$-stable} since they are minimal components of the zero ideal, which is {$D_{R|\ZZ}(R,R)$-stable}, and thus so is their sum $(x,y)$. If $c>0$, then $\gamma(px) = p \gamma(x) \in pQ\subseteq Q^2$, and $a+b<2$, so $b \leqslant 1$ and thus 
\[{\alpha \delta^b \gamma(px)\in Q}.\] 
If $b=0$, the claim is clear by $\ZZ$-linearity of elements of $D_{R|\ZZ}(R,R/Q)$. Otherwise, note that 
\[{\delta(px) = \Phi(p) \delta(x) + x^p \delta(p) = x^p \delta(p)}.\] 
Then, in the case $b=2$ we have 
\[\delta(\delta(px)) = \Phi(\delta(p)) \delta(x^p) + x^{p^2} \delta(\delta(p)) \in (x^{p^2}) \subseteq Q. \]
In the case $a=b=1$, for any $\alpha\in D^1_{R|\ZZ}(R,R/Q)$ we have 
\[\alpha(\delta(px))={\alpha(x^p \delta(p)) \in (x^{p-1}) \subseteq Q}.\] 
This shows the claim.
\end{example}

\section{Unramified discrete valuations rings with no $p$-derivations}
The following example is based on examples of discrete valuation rings of equal characteristic with no nontrivial derivations due to Rotthaus \cite[p.~319]{Rotthaus} and Mukhopadhyay and Smith \cite{MukhopadhyaySmith}. We adapt their arguments to show that there exists an unramified discrete valuation ring which has no $p$-derivations. 

Let $A=\ZZ_{(2)}[x]/(x^2+x+1)$, which is an unramified discrete valuation ring with uniformizer~$2$. Let $\widehat{(-)}$ denote $2$-adic completion, and take $V=\widehat{A}$. Let $L$ be the fraction field of~$A$.  Note that both $L$ and $V$ live inside the fraction field $K$ of $V$. Given an element $w \in K$ we let $R_w = L(w) \cap V$. We note that $R_w$ is an unramified discrete valuation ring with uniformizer~$2$, since it is the valuation ring of the restriction of the $2$-adic valuation on~$V$ from~$K$ to $L(w)$.

\begin{theorem}\label{bad dvr} 
With the notation introduced above, there exists $w \in K$ such that $R_w$ is an unramified discrete valuation ring which has no $2$-derivations.
\end{theorem} 

\begin{proof} 
Let us for now consider any element $w \in K$, and let $\Phi$ be a lift of Frobenius on~$R_w$. Since $A \subseteq R_w \subseteq V$, then $\Phi$ induces a map $A \to V$ which, in turn, induces a map $V \to V$ that is still a lift of Frobenius. However, the only lift of Frobenius on $V$ is the one which is the identity on $\ZZ_2 \colonequals \widehat{\ZZ_{(2)}}$ and such that $\Phi(x)=x^2$. For a direct proof of this, note first that the knowledge of the restriction of $\Phi$ to $\ZZ_2$  and the value of $\Phi(x)$ is indeed sufficient to determine $\Phi$ on $V$. The unique lift of Frobenius on $\ZZ_2$ is the identity. If $\Phi(x) = x^2+2\delta(x)$, then we must have
\[(x^2+2\delta(x))^2+(x^2+2\delta(x)) + 1 = 0\]
which implies
\[\delta(x)\left(1+2x^2+2\delta(x)\right)=0.\]
Note that $1 \notin (2)$ and thus $1+2x^2+2\delta(x) \neq 0$, therefore we must have $\delta(x) = 0$.

Now we choose $w$ as follows: since $L$ is countable, and $\ZZ_2$ is uncountable, we can find $u,v \in \ZZ_2$ such that ${\rm tr.deg}_L(L(u,v)) = 2$, and we take $w=u+xv$. If $R_w$ had a $2$-derivation or, equivalently, a lift of Frobenius $\Phi$, then by what we proved above we would have 
\[
\Phi(u)=u, \quad \Phi(v)=v, \quad \Phi(x)=x^2 .\]
The condition that $\Phi(w) = u+x^2v \in R_w \subseteq L(w)$ gives 
\[ L(w) = L(u+xv,u+x^2v) = L(u,v).\] 
This contradicts our choice of $u$ and $v$, since $2={\rm tr.deg}_L(L(u,v))={\rm tr.deg}_L(L(w)) \leq 1$.
\end{proof}

\

\subsection*{Acknowledgments} The first author was partially supported by the MIUR Excellence Department Project CUP D33C23001110001, PRIN 2022 Project 2022K48YYP, and by INdAM-GNSAGA. The second author was partially supported by NSF grant DMS-2236983. The third author was partially supported by NSF grant DMS-2044833. The second and third authors thank SECIHTI/CONAHCYT, Mexico, for its support with grant CF-2023-G-33.

This material is based upon work supported by the National Science Foundation under Grant No. DMS-1928930 and by the Alfred P. Sloan Foundation under grant G-2021-16778, while the three authors were in residence at the Simons Laufer Mathematical Sciences Institute (formerly MSRI) in Berkeley, California, during the Spring 2024 semester.

\bibliographystyle{alpha}
\bibliography{References}

\end{document}